\documentclass[letterpoint,12pt]{article}
\usepackage{amsmath,amssymb,amsthm}
\usepackage{graphicx}
\usepackage{multicol}

\newtheorem{theorem}{Theorem}

\newtheorem{lemma}[theorem]{Lemma}

\newtheorem{observation}[theorem]{Observation}

\newtheorem{question}[theorem]{Question}

\begin{document}
\title{A graph for which the inertia bound is not tight}
\author{John Sinkovic\thanks{Department of Combinatorics and Optimization, University of Waterloo, ({\tt johnsinkovic@gmail.com})}
}
\date{\today}
\maketitle

\begin{abstract}
The inertia bound gives an upper bound on the independence number of a graph by considering the inertia of matrices corresponding to the graph.  The bound is known to be tight for graphs on 10 or fewer vertices as well as for all perfect graphs.  The question has been asked as to whether the bound is always tight.  We show that the bound is not tight for the Paley graph on 17 vertices as well as for the graph obtained from Paley 17 by deleting a vertex.
\end{abstract}

\bigskip


\bigskip

\bigskip

\section{Introduction}\label{sec:intro}
In algebraic graph theory it is common to define a matrix or family of matrices using a simple graph.  Properties of a such a matrix or matrices are then analyzed to make connections to a property of the graph.  One such connection exists between the independence number $\alpha(G)$ and the eigenvalues of weight matrices $W$ corresponding to $G$.  A weight matrix $W$ of $G$ is a real symmetric matrix such that $w_{ij}=0$ if $ij$ is not an edge of $G$.  For any weight matrix $W$ of $G$, 
$$\alpha(G)\leq\min\{|G|-n_+(W),|G|-n_-(W)\}.$$
(See Lemma 9.6.3 in \cite{godsilroyle}.) This upper bound on $\alpha$ is attributed to D.M. Cvetkovi\'c (PhD thesis 1971).  It is often referred to as the Cvetkovi\'c bound or inertia bound.

As with many inequalities, it is interesting to determine when equality occurs or in other words when the bound is tight.  The following question was posed by Chris Godsil\footnote{\emph{Interesting Graphs and their Colourings}, unpublished lecture notes C. Godsil (2004)} and appears in \cite{elzinga,gregory,rooney}.

\begin{question}\label{qu:the question} Does each graph $G$ have a weight matrix $W$  such that \center $\alpha(G)=\min\{|G|-n_+(W),|G|-n_-(W)\}$?
\end{question}

The inertia bound has been shown to be tight for small graphs \cite{elzinga, gregory} (all graphs on 10 or fewer vertices, vertex-transitive graphs on 12 or fewer vertices), certain families of Cayley and strongly regular graphs \cite{rooney}, and perfect graphs\footnote{\emph{Interesting Graphs and their Colourings}, unpublished lecture notes C. Godsil (2004)}.  The smallest vertex-transitive graph for which the tightness of the bound has not been determined is Paley 13 (\cite{elzinga, gregory, rooney}).  Paley 17 is the first graph for which a proof has been given showing that the inertia bound is not tight.

\section{Preliminaries}\label{sec:prelim}
Let $G=(V,E)$ be a simple graph of order $n$ with vertex set $V$ and edge set $E$.  A {weight matrix} $W$ of $G$ is an $n \times n$ real, symmetric matrix such that $w_{ij}=0$ if $(i,j)\not\in E(G)$.
As $G$ is simple, it has no loops, and $w_{ii}=0$ for all $i$.  In example, any multiple of the adjacency matrix of $G$ is a weight graph for $G$.  Let $n_+(W), n_-(W)$, and $n_0(W)$ denote the number of positive eigenvalues, the number of negative eigenvalues, and the multiplicity of zero as an eigenvalue of $W$, respectively.

Given any weight matrix $W$ of $G$, the Cvetkovi\'c bound or inertia bound states that\begin{equation}\label{eq:inertiabound}\alpha(G)\leq \min\{n-n_+(W),n-n_-(W)\}.\end{equation}
The following is an equivalent form of the inertia bound which will be useful to our discussion.  For any weight matrix $W$ of $G$,
\begin{equation}\label{eq:altinertiabound}\alpha(G)\leq n_0(W)+\min\{n_+(W),n_-(W)\}.
\end{equation} 

As our goal is to exhibit a graph $G$ for which the bound is not tight, we point out that in Equation (\ref{eq:altinertiabound}) equality cannot occur when both $n_+(W)$ and $n_-(W)$ are greater than $\alpha(G)$.   The claim is that all weight matrices for Paley 17 have at least 4 positive and 4 negative eigenvalues while the independence number is 3.  We prove the claim by looking at invertible principal submatrices corresponding to induced subgraphs of order 7.  Interlacing plays a big part in our proof and as such we include this well-known theorem.

\begin{theorem}\label{thm:alt interlacing}(Theorem 9.1.1 of \cite{godsilroyle}) Let $A$ be an $n\times n$ real symmetric matrix with eigenvalues $\lambda_1\geq \lambda_2\geq\ldots\geq\lambda_n$ and let $C$ be a $k\times k$ principal submatrix of $A$ with eigenvalues $\tau_1\geq \tau_2\geq \ldots\geq \tau_k$. Then $\lambda_i\geq \tau_i$ for all $i\in \{1,\ldots,k\}$.
\end{theorem}

 A weight matrix $W$ of $G$ is $\emph{optimal}$ if $\alpha(G)=\min\{|G|-n_+(W),|G|-n_-(W)\}$. 

\begin{lemma}\label{lem:4iscontradicting} Let $W$ be a weight matrix for a graph $G$ of order $n$.  If $W$ has a principal submatrix with $\alpha(G)+1$ positive eigenvalues and a principal submatrix with $\alpha(G)+1$ negative eigenvalues, then $W$ is not optimal.
\end{lemma}

\begin{proof} By Theorem \ref{thm:alt interlacing}, $W$ has at least $\alpha(G)+1$ positive and $\alpha(G)+1$ negative eigenvalues.    So $n_-(W)\leq n-\alpha(G)-1$ and $n_+(W)\leq n-\alpha(G)-1$. Thus $\min\{n-n_+(W),n-n_-(W)\}\geq \alpha(G)+1>\alpha(G)$. Therefore Equation (\ref{eq:inertiabound}) is not tight and $W$ is not optimal.

\end{proof}

The independence number of Paley 17 is 3 (Observation \ref{obs:alphapaley17}).  We will use Lemma \ref{lem:4iscontradicting} to show that the inertia bound is not tight for Paley 17.  To do so, we need to show that every weight matrix for Paley 17 has principal submatrices which contain 4 positive eigenvalues and 4 negative eigenvalues.

\section{Paley 17}\label{sec:paley} 
We now pause for a moment to consider properties of Paley 17.  Paley graphs were first studied in connection with number theory and have been used in giving lower bounds for diagonal Ramsey numbers. The following definition is taken from \cite{godsilroyle}. Henceforth $q$ is a prime power such that $q\equiv 1\pmod 4$.  The \emph{Paley graph} $P(q)$ has as vertex set the elements of the finite field $GF(q)$, with two vertices adjacent if and only if their difference is a nonzero square in $GF(q)$.

Paley graphs are (among other things) strongly regular graphs, self-complementary, circulants, vertex-transitive, edge transitive, and arc-transitive (See \cite{bollabas} chapter XIII).

An \emph{independent set} or \emph{coclique} is a subset of vertices which are pairwise nonadjacent.
The size of a maximum independent set is the \emph{independence number} of $G$ and is denoted $\alpha(G)$.   The independence number for a large number of Paley graphs has been calculated by James B. Shearer\footnote{\emph{Independence number of Paley graphs--IBM Research}, Data published online at http://www.research.ibm.com/people/s/shearer/indpal.html}.
\begin{figure}[h]
\caption{Paley 17}
\includegraphics[scale=.5]{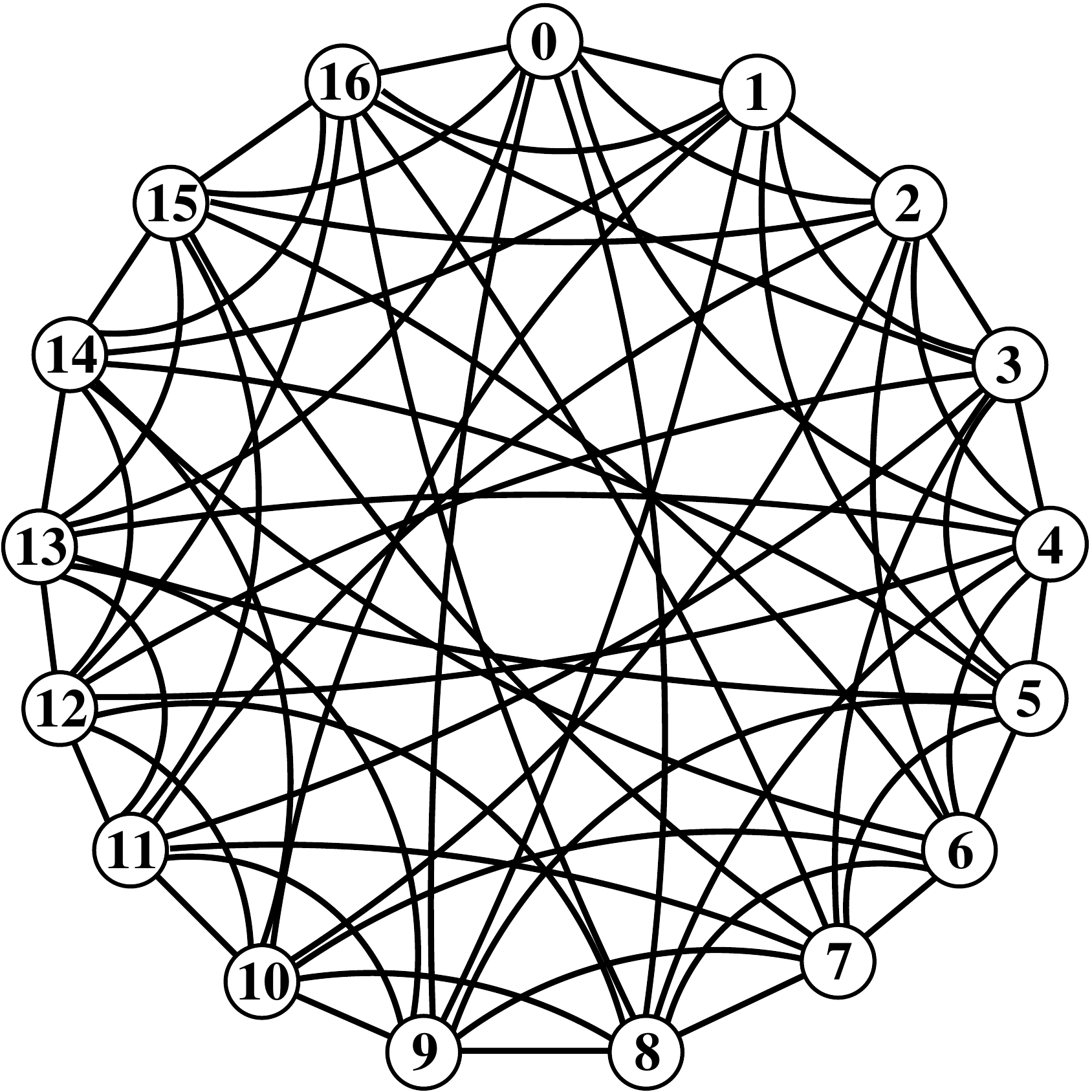}
\centering
\end{figure}
\begin{observation}\label{obs:alphapaley17} The independence number of $P(17)$ is 3.
\end{observation}

A graph $G$ is \emph{$\alpha$-critical} if $\alpha(G)<\alpha(G-e)$ for all edges $e$.

\begin{observation}\label{obs:alphacritical} The graph $P(17)$ is $\alpha$-critical.
\end{observation}  

\begin{proof} Since $P(17)$ is arc-transitive it is sufficient to show that deleting edge $(0,1)$ increases $\alpha$.  The set $\{0,1,7,12\}$ is independent in $P(17)$ delete $(0,1)$.
\end{proof}

An \emph{automorphism} of a graph $G$ is a permutation $\sigma$ of $V(G)$ such that $v_1\sim v_2$ if and only if $\sigma(v_1)\sim \sigma(v_2)$.

\begin{lemma}\label{lem:automorphism} The function $\sigma_{ab}:V\rightarrow V$, $\sigma_{ab}(v)=av+b$ 
where $a$ is a nonzero square in $GF(q)$ and $b\in GF(q)$ is an automorphism of $P(q)$.
\end{lemma}
\begin{proof} Note that $\sigma_{ab}$ is one-to-one and onto.  Given two vertices $v_1,v_2$ of $P(q)$, a nonzero square $a$ in $GF(q)$, and their images $\sigma_{ab}(v_1),\sigma_{ab}(v_2)$.  Now $v_1\sim v_2$ if and only if $v_1-v_2$ is a nonzero square.  Since $a$ is a nonzero square, $v_1-v_2$ is a nonzero square if and only if $a(v_1-v_2)$ is a nonzero square.  Further, $a(v_1-v_2)$ is a nonzero square if and only if $av_1-b-(av_2-b)$ is a nonzero square if and only if $\sigma_{ab}(v_1)\sim \sigma_{ab}(v_1)$.

\end{proof}

The quadratic residues (or nonzero squares) modulo 17 are $\pm1,\pm 2,\pm 4, \pm 8$.  A $k$-\emph{edge} of $P(17)$ is an edge $v_1v_2$ such that $v_1-v_2=\pm k\pmod {17}$ for $k\in \{1,2,4,8\}$.   For a given $k$, the set of $k$ edges form a 17 cycle.  The 1-edges form the cycle $0$, $1$, $2$, $3,\ldots,16$, $0$, the 2-edges form the cycle $0,2,4,\ldots,$ $14,16,1,3,\ldots,$ $13,15,0$, the 4-edges form the cycle $0$, $4$, $8$, $12$, $16$, $3$, $7$, $11$, $15$, $2$, $6$, $10$, $14$, $1$, $5$, $9$, $13$, $0$, and the 8-edges form the cycle $0$, $8$, $16$, $7$, $15$, $6$, $14$, $5$, $13$, $4$, $12$, $3$, $11$, $2$, $10$, $1$, $9$, $0$.  Thus we have the following observation.

\begin{observation}\label{obs:2factorization} $P(17)$ has a 2-factorization consisting of four cycles of length 17.
\end{observation}

An $a$-$b$-$c$ triangle is a triangle in $P(17)$ which consists of an $a$-edge, a $b$-edge, and a $c$-edge.  For example, $\Delta(0,1,2)$ is a 1-1-2 triangle, while $\Delta(0,8,9)$ is an 8-8-1 triangle.

\begin{lemma}\label{lem:autoacttranstriangles}There exists an automorphism of $P(17)$ which maps $\Delta(0,1,2)$ to any other triangle. In other words, the group of automorphisms of $P(17)$ acts transitively on its triangles.\end{lemma}

\begin{proof} The graph $P(17)$ has 68 triangles as seen from the characteristic polynomial (coefficient of $t^{14}$ is $-136$). There are 17 1-1-2 triangles, 17 2-2-4 triangles, 17 4-4-8 triangles, and 17 8-8-1 triangles.  The automorphisms $\sigma_{1b}$ for $b\in\{0,1,\ldots, 16\}$ map $\Delta(0,1,2)$ to any 1-1-2 triangle.  Similarly $\sigma_{ab}$ for $a\in\{2,4,8\}$, $b\in\{0,1,\ldots, 16\}$ maps $\Delta(0,1,2,)$ to any 2-2-4, 4-4-8, or 8-8-1 triangle.

\end{proof}

Lemma \ref{lem:autoacttranstriangles} will be used to show that there are many isomorphic copies of the induced subgraphs of Section \ref{sec:indsub}.

\section{Induced Subgraphs of $P(17)$}\label{sec:indsub}

In this section we consider two induced subgraphs $G_1$ and $G_2$ of $P(17)$.  When the entries of the principal submatrices corresponding to the edges of $G_1$ and $G_2$ are nonzero, their determinants are nonzero.  This fact and the inertia bound are used to show that such principal submatrices have either 3 positive and 4 negative eigenvalues, or 4 positive and 3 negative eigenvalues.

Let $G_1$ be the following graph on 7 vertices and $W_1$ a weight matrix for $G_1$.

\begin{multicols}{2}
\begin{center}
\includegraphics[scale=.5]{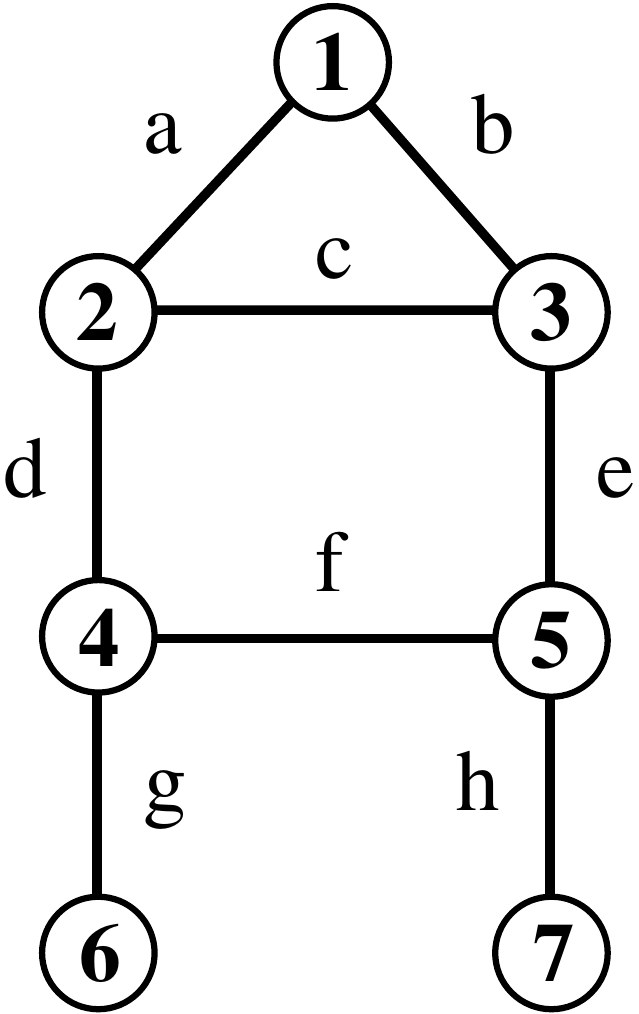}
\end{center}

\begin{center}
$W_1=\left[\begin{array}{ccccccc}   
0 & a & b & 0 & 0 & 0 & 0\\   
a & 0 & c & d & 0 & 0 & 0\\  
b & c & 0 & 0 & e & 0 & 0\\ 
0 & d & 0 & 0 & f & g & 0\\ 
0& 0 & e & f & 0 & 0 & h\\ 
0 & 0 & 0 & g & 0 & 0 & 0\\ 
0 & 0 & 0 & 0 & h & 0 & 0\\ 
\end{array}\right]$
\end{center}

\end{multicols}

\begin{observation}\label{obs:nonzerodet} $\det(W_1)=2abcg^2h^2$
\end{observation}
\begin{observation}\label{ons:indsub1} The graph $G_1$ is an induced subgraph of Paley 17.
\end{observation}

\begin{center}

\includegraphics[scale=.5]{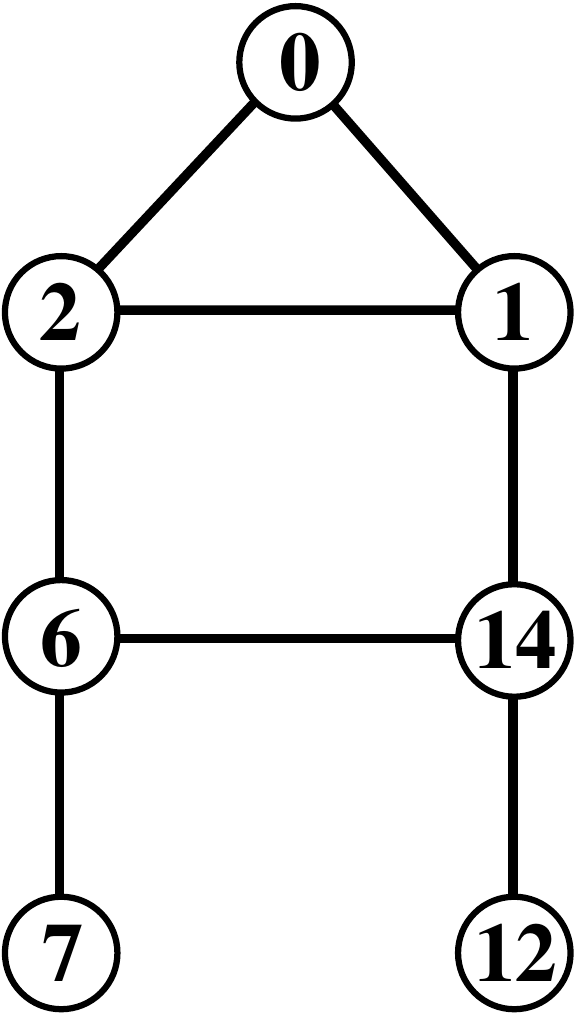}
\hspace{2cm}
\includegraphics[scale=.3]{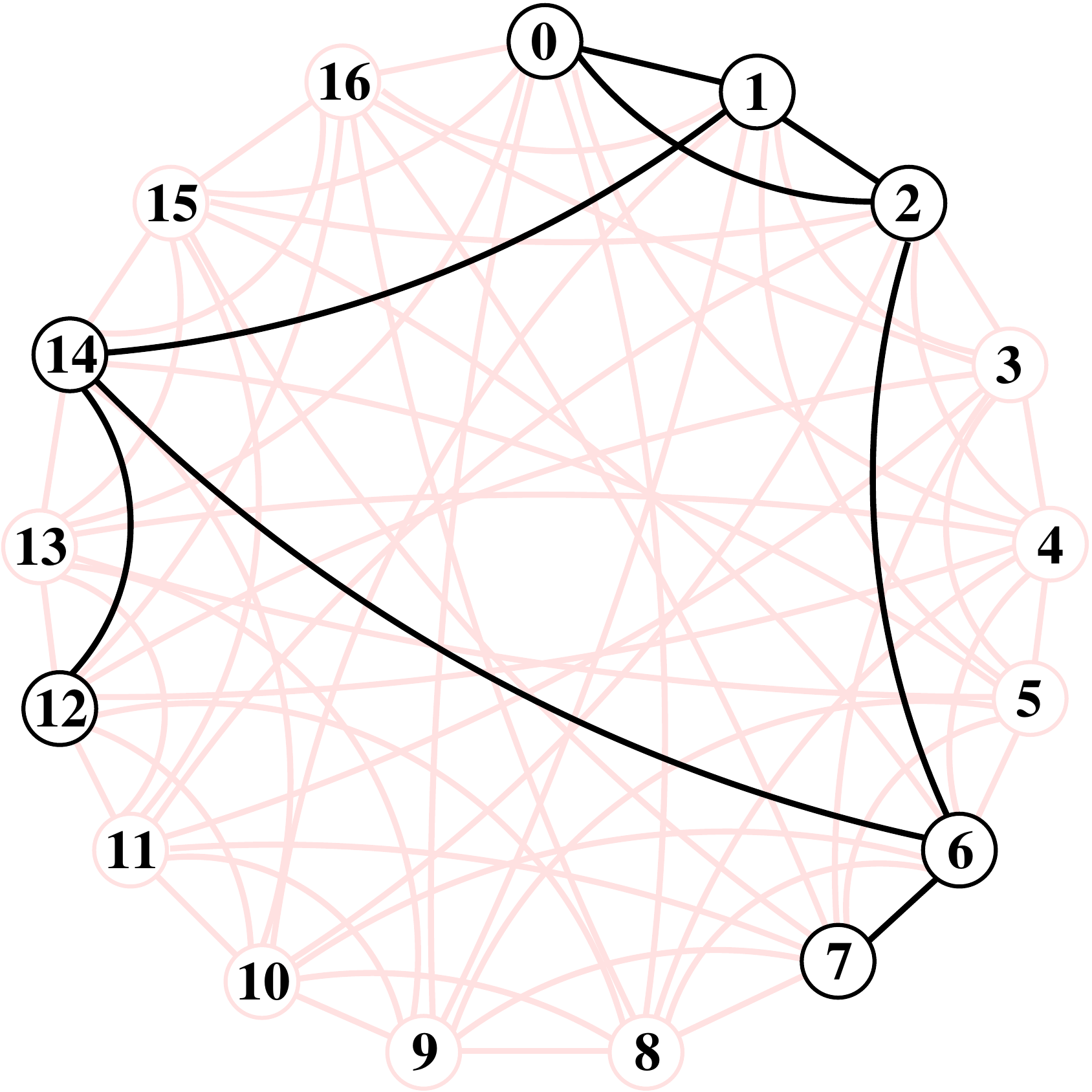}
\hfill
\end{center}
\begin{lemma}\label{lem:posnegevals} Let the product $abcgh\neq 0$.  
\begin{itemize} 
\item If $abc>0$ then $n_+(W_1)=3$ and $n_-(W_1)=4$.
\item If $abc<0$ then $n_+(W_1)=4$ and $n_-(W_1)=3$. 
\end{itemize}
\end{lemma}
\begin{proof}  Apply Equation (\ref{eq:inertiabound}) to $G_1$ and $W_1$.  Since $\alpha(G_1)=3$, we have that $$3\leq \min\{7-n_+(W_1),7-n_-(W_1)\}.$$ Thus $n_-(W_1)\leq 4$ and $n_+(W_1)\leq 4$.  Since $abcgh\neq 0$, $\det(W_1)\neq 0$. Thus either $n_+(W_1)=4$ and $n_-(W_1)=3$ or $n_+(W_1)=3$ and $n_-(W_1)=4$.  Further, if $abc>0$, then $\det(W_1)>0$ so the product of the eigenvalues is positive and $n_-(W_1)=4$.  On the other hand, if $abc<0$, then $\det(W_1)<0$ and $n_-(W_1)=3$.

\end{proof}

Note that Lemma \ref{lem:posnegevals} establishes the relationship between the sign of a triangle in $P(17)$ and whether a principal submatrix has 4 positive or 4 negative eigenvalues.  In Section \ref{sec:weight matrices} we will use Lemma \ref{lem:autoacttranstriangles} to show that every triangle is part of an isomorphic copy of $G_1$.  We then deduce that in an optimal weight matrix for $P(17)$ all triangles have the same sign.

Let $G_2$ be the following graph on 7 vertices and $W_2$ a weight matrix for $G_2$.

\begin{multicols}{2}
\begin{center}
\includegraphics[scale=.4]{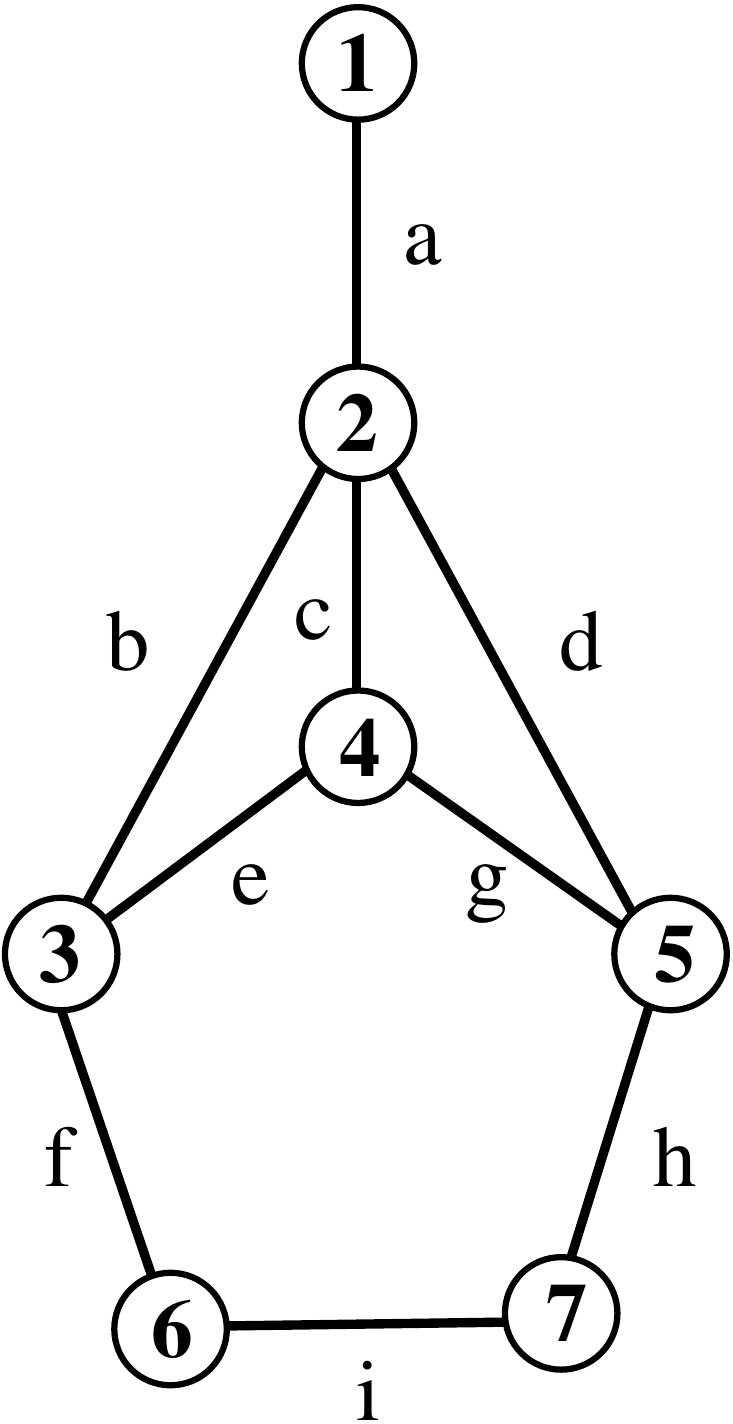}
\end{center}

\begin{center}
$W_2=\left[\begin{array}{ccccccc}   
0 & a & 0 & 0 & 0 & 0 & 0\\   
a & 0 & b & c & d & 0 & 0\\  
0 & b & 0 & e & 0 & f & 0\\ 
0 & c & e & 0 & g & 0 & 0\\ 
0 & d & 0 & g & 0 & 0 & h\\ 
0 & 0 & f & 0 & 0 & 0 & i\\ 
0 & 0 & 0 & 0 & h & i & 0\\ 
\end{array}\right]$
\end{center}

\end{multicols}

\begin{observation}\label{obs:nonzerodet2} $\det(W_2)=-2a^2efghi$
\end{observation}

\begin{observation}\label{obs:indsub2} The graph $G_2$ is an induced subgraph of Paley 17.
\end{observation}

\begin{center}

\includegraphics[scale=.4]{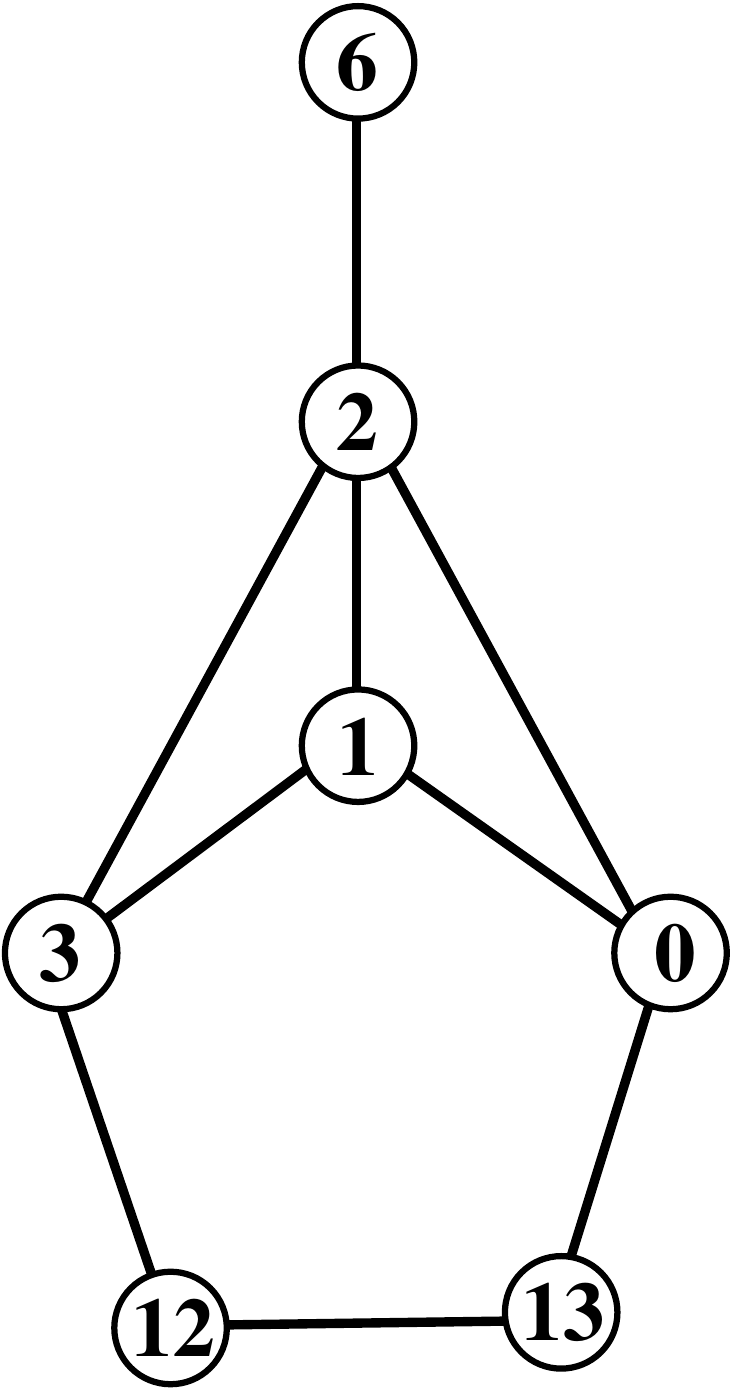}
\hspace{2cm}
\includegraphics[scale=.3]{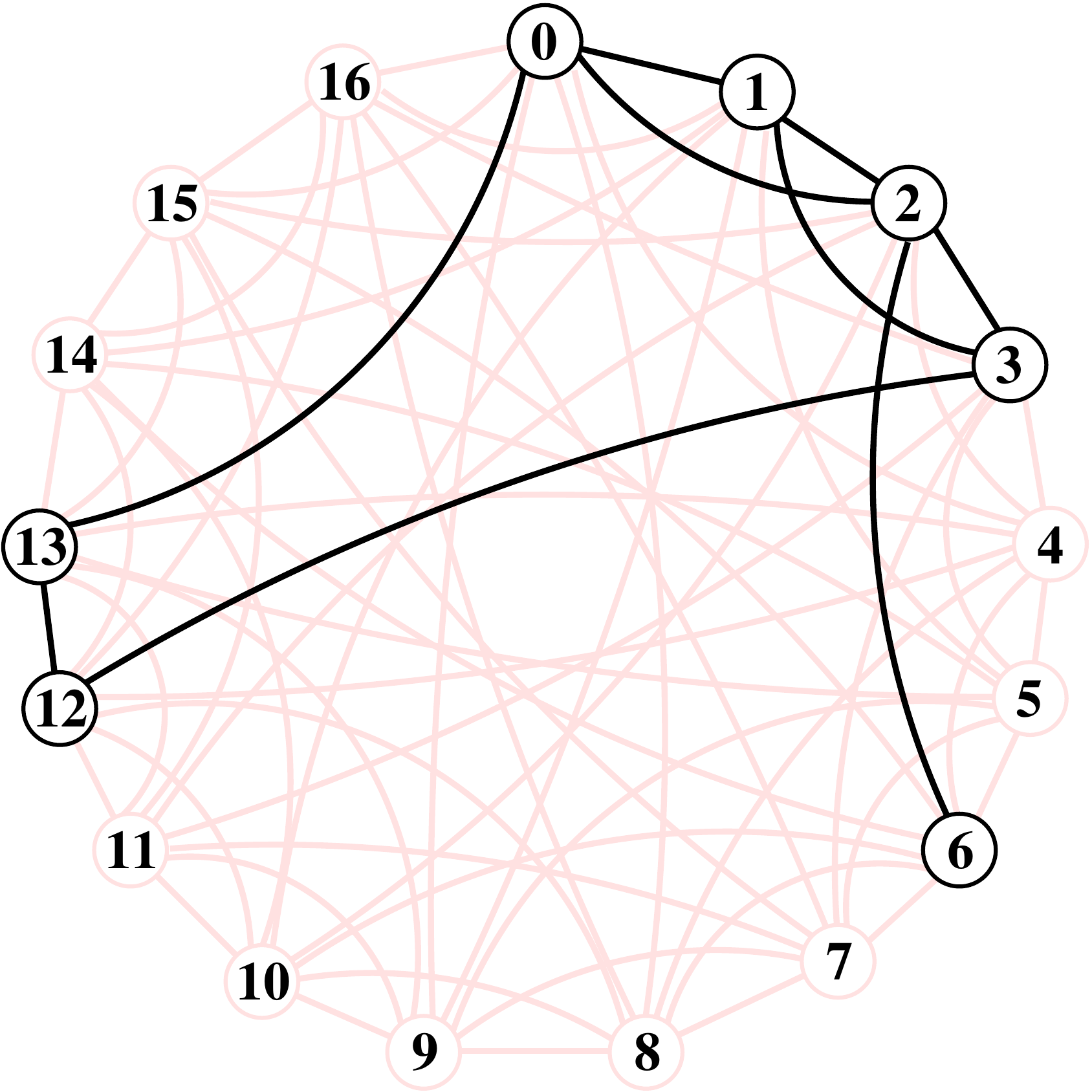}
\hfill
\end{center}

\begin{lemma}\label{lem:posnegevals2} Let the product $aefghi\neq 0$.  
\begin{itemize} 
\item If $efghi<0$, then $n_+(W_2)=3$ and $n_-(W_2)=4$.
\item If $efghi>0$, then $n_+(W_2)=4$ and $n_-(W_2)=3$. 
\end{itemize}
\end{lemma}
\begin{proof}  Apply Equation (\ref{eq:inertiabound}) to $G_2$ and $W_2$.  Since $\alpha(G_2)=3$, we have that $$3\leq \min\{7-n_+(W_2),7-n_-(W_2)\}.$$ Thus $n_-(W_2)\leq 4$ and $n_+(W_2)\leq 4$.  Since $aefghi\neq 0$, $\det(W_2)\neq 0$. Thus either $n_+(W_2)=4$ and $n_-(W_2)=3$ or $n_+(W_2)=3$ and $n_-(W_2)=4$.  Further, if $efghi<0$, then $\det(W_2)>0$.  So the product of the eigenvalues is positive and $n_-(W_2)=4$.  On the other hand, if $efghi>0$, then $\det(W_2)<0$ and $n_-(W_1)=3$.

\end{proof}

\section{Optimal weight matrices}\label{sec:weight matrices}

We now continue with the strategy stated in Section \ref{sec:prelim} by considering optimal weight matrices of $P(17)$.  Recall, a weight matrix $W$ of $G$ is $\emph{optimal}$ if $\alpha(G)=\min\{n-n_+(W),n-n_-(W)\}$.  The following observation can be found in \cite{elzinga} and follows by using the Cvetkovi\'c bound on $G-e$.

\begin{observation}\label{obs:noedgedeletion}(Lemma 7.2 in \cite{elzinga}) Let $G$ be $\alpha$-critical and $W$ an optimal weight matrix of $G$. Then $w_{ij}\neq 0$ for all $ij\in E(G)$.
\end{observation}

In light of Observations \ref{obs:noedgedeletion} and \ref{obs:alphacritical} we have the following observation.

\begin{observation}\label{obs:nonzero} Any optimal weight matrix $W$ of $P(17)$ has $w_{ij}\neq 0$ if $ij$ is an edge of $P(17)$.
\end{observation}

Recall that entries of weight matrices corresponding to edges are not necessarily nonzero.  However, Observation \ref{obs:nonzero} states that edges correspond to nonzero entries in optimal weight matrices for Paley 17.  This necessary fact will allow the use of Lemmas \ref{lem:posnegevals} and \ref{lem:posnegevals2}.

Given a graph $G$ and a weight matrix $W$, the \emph{sign} of a triangle is the sign of the product of the entries of $W$ corresponding to edges of the triangle of $G$.

\begin{lemma}\label{lem:triangleshavesamesign} Let $W$ be an optimal weight matrix for $P(17)$.  Every triangle of $P(17)$ has the same sign.
\end{lemma}

\begin{proof} Let $W$ be an optimal weight matrix for $P(17)$. Suppose by way of contradiction that there exists a triangle $\Delta_1$ which has a different sign than $\Delta(0,1,2)$. Without loss of generality, negating $W$ if necessary, suppose that the sign of $\Delta(0,1,2)$ is negative. By Lemma \ref{lem:autoacttranstriangles}, there exists an automorphism $\sigma$ of $P(17)$ which maps $\Delta(0,1,2)$ to $\Delta_1$ .  Thus $\sigma$ maps $G_1$ to $H_1$, an isomorphic copy of $G_1$ containing $\Delta_1$.  By Observation \ref{obs:nonzero} the entries of $W$ corresponding to edges are nonzero. Thus Lemma \ref{lem:posnegevals} and the fact that the signs of the triangles are different imply that the principal submatrix of $W$ corresponding to $G_1$ has 4 positive eigenvalues while the principal submatrix of $W$ corresponding to $H_1$ has 4 negative eigenvalues.  By  Lemma \ref{lem:4iscontradicting} and Observation \ref{obs:alphapaley17}, $W$ is not optimal. Contradiction. Therefore all triangles have the same sign.
\end{proof}

\begin{lemma}\label{lem:path} Let $A$ be a symmetric $n\times n$ matrix with nonzero sub and super diagonal.   There exists a $(1,-1)$-diagonal matrix $D$ such that the sub and super diagonals of $DAD$ are positive.
\end{lemma}

\begin{proof} We proceed by induction on $n$.  If $n=2$ and $a_{12}=a_{21}> 0$ we are done. So assume $a_{12}<0$. Let $D$ be a diagonal matrix with exactly one 1 and one -1.  Then the off-diagonal entries of $DAD$ are positive.

Now assume that the hypothesis is true for $n=k$.  Let $A$ be a symmetric matrix of order $k+1$  with nonzero sub and super diagonals.    Since $A(k+1)$ ($A$ with $k+1$th row and column deleted) is a $k\times k$ symmetric matrix, there exists a $(1,-1)$-diagonal matrix $D_1$ such that the sub and super diagonal entries of $D_1A(k+1)D_1$ are positive. If $a_{k,k+1}=a_{k+1,k}>0$, then $D=D_1\oplus [1]$.  On the other hand, if $a_{k,k+1}=a_{k+1,k}<0$ then $D=D_1\oplus[-1]$.  Then the sub and super diagonal entries of $DAD$ will be positive.
\end{proof}

Note that in Lemma \ref{lem:path}, $D^2$ is the identity matrix.  Thus $DAD$ and $A$ are similar and have the same eigenvalues.  In light of this fact, we may assume that any weight matrix for $P(17)$ has positive entries on its super and sub diagonal.  These entries correspond to all the 1-edges of $P(17)$ except for edge $(16,0)$.  So we will have two cases to consider.  One case where all the 1-edges are positive and one where all the 1-edges are positive except for edge $(16,0)$ which is negative.

\begin{lemma}\label{lem:allpos} Let $W$ be an optimal weight matrix of $P(17)$ and assume that all the 1-edges are positive.  Then all edges of $P(17)$ are positive and the sign of every triangle is positive.
\end{lemma}

\begin{proof} Let $W$ be an optimal weight matrix of $P(17)$.  By Lemma \ref{lem:triangleshavesamesign}, all the triangles of $P(17)$ have the same sign.  Suppose that the sign of every triangle is negative. Since every 2-edge belongs to a 1-1-2 triangle and all 1-edges are positive, all the 2-edges are negative.  Since every 4-edge belongs to a 2-2-4 triangle and the 2-edges are negative, all the 4-edges are negative.  Since every 8-edge belongs to a 4-4-8 triangle and the 4-edges are negative, the 8-edges are negative.  But $\Delta(0,8,9)$ is an 8-8-1 triangle which has positive sign.  Thus it is not possible for the 1-edges to be positive and have every triangle be negative.

Thus the sign of every triangle is positive.  Following a similar argument as above, 1-edges force 2-edges to be positive which force 4-edges to be positive which force 8-edges to be positive. Thus all the edges of $P(17)$ are positive as desired.
\end{proof}

\begin{lemma}\label{lem:oneneg} Let $W$ be an optimal weight matrix of $P(17)$ and assume that all the 1-edges are positive with the exception of edge $(16,0)$.  Then the remaining edges are negative with the exception of 2-edges $(1,16)$, $(0,15)$, 4-edges $(0,13)$, $(1,14)$, $(2,15)$, $(3,16)$, and 8-edges $(0,9)$, $(1,10)$, $(2,11)$, $(3,12)$, $(4,13)$, $(5,14)$, $(6,15)$, $(7,16)$.  Further, the sign of every triangle is negative.
\end{lemma}

\begin{proof} Let $W$ be an optimal weight matrix of $P(17)$.  By Lemma \ref{lem:triangleshavesamesign}, all the triangles of $P(17)$ have the same sign.  Suppose that the sign of every triangle is positive.  Since every 2-edge belongs to a 1-1-2 triangle and all 1-edges except $(0,16)$ are positive, all 2-edges are positive with the exception of edge $(1,16)$ and $(0,15)$.  Since every 4-edge belongs to a 2-2-4 triangle and all the 2-edges are positive except $(1,16)$ and $(0,15)$, all 4-edges are positive with the exception of $(0,13),(1,14),(2,15),(3,16)$.
Since every 8-edge belongs to a 4-4-8 traingle and all the 4-edges are positive except $(0,13),(1,14),(2,15),(3,16)$, all the 8-edges are positive with the exception of edges $(0,9),(1,10),(2,11),(3,12),(4,13),(5,14),(6,15),(7,16)$.  Consider the 8-8-1 triangle, $\Delta(0,8,9)$.  Edge $(0,9)$ is the only negative edge of the triangle and thus $\Delta(0,8,9)$ has negative sign.  This contradicts that the sign of every triangle is positive.

Thus the sign of every triangle is negative.  Note that changing the sign of the triangles from positive to negative just switches the sign of all the 2-edges, 4-edges, and 8-edges in the above argument.  Therefore all the 2-edges are negative with the exception of $(1,16),(0,15)$, all the 4-edges are negative with the exception of $(0,13),(1,14),(2,15),(3,16)$, and all the 8-edges are negative with the exception of $(0,9),(1,10),(2,11),(3,12),(4,13),(5,14),(6,15)$, $(7,16)$.
\end{proof}

\begin{theorem}\label{thm:nottight} The Paley graph on 17 vertices has no optimal weight matrices.
\end{theorem}

\begin{proof}  Suppose by way of contradiction that $W$ is an optimal weight matrix of Paley 17. By Lemma \ref{lem:path}, there exists a diagonal matrix $D$ such that the sub and super diagonal entries of $DWD$ are positive.  As $D$ is a diagonal matrix the zero-nonzero structure of $W$ is maintained.  Further, as $D^2$ is the identity, $DWD$ is similar to $W$.  Thus $DWD$ is an optimal weight matrix for $P(17)$.  Let $DWD=\widetilde{W}$.  Notice that the sub and super diagonal entries of $\widetilde{W}$ correspond to the 1-edges of $P(17)$ except $(0,16)$.  As we will be using Lemmas \ref{lem:posnegevals} and \ref{lem:posnegevals2}, we point out that by Observation \ref{obs:nonzero} all the entries of $\widetilde{W}$ corresponding to edges are nonzero.

First we consider the case where $(0,16)$ is positive.  By Lemma \ref{lem:allpos}, all the edges of $P(17)$ are positive.  In other words all the entries of $\widetilde{W}$ corresponding to edges are positive.  By Observation \ref{obs:indsub2}, $G_2$ is an induced subgraph of $P(17)$.  By Lemma \ref{lem:posnegevals2}, $\widetilde{W}$ has a principal submatrix corresponding to $G_2$ which has four positive eigenvalues.  On the other hand, Lemmas \ref{lem:allpos} and \ref{lem:posnegevals} imply that $\widetilde{W}$ has a principal submatrix corresponding to $G_1$ which has four negative eigenvalues.  

Now the case where $(0,16)$ is negative and the rest of the 1-edges are positive. By Lemma \ref{lem:oneneg}, the remaining edges are negative with the exception of 2-edges $(1,16)$, $(0,15)$, 4-edges $(0,13)$, $(1,14)$, $(2,15)$, $(3,16)$, and 8-edges $(0,9)$, $(1,10)$, $(2,11)$, $(3,12)$, $(4,13)$, $(5,14)$, $(6,15)$, $(7,16)$.   Consider the edges of the five cycle in the induced subgraph $G_2$ with vertices $0$, $1$, $2$, $3$, $6$, $12$, and $13$.  Edges $(0,1)$, $(0,13)$, $(13,12)$,$(12,3)$ are positive, while edge $(1,3)$ is negative.  Thus by Lemma \ref{lem:posnegevals2}, the principal submatrix of $\widetilde{W}$ corresponding to $G_2$ has four negative eigenvalues.  On the other hand, Lemmas \ref{lem:oneneg} and \ref{lem:posnegevals} imply that $\widetilde{W}$ has a principal submatrix corresponding to $G_1$ which has four positive eigenvalues.

Thus in either case $\widetilde{W}$ has principal submatrices which have four positive and four negative eigenvalues.  Recall (Observation \ref{obs:alphapaley17}) that $\alpha(P(17))=3$.   By Lemma \ref{lem:4iscontradicting}, $\widetilde{W}$ is not optimal.  Contradiction. Thus there does not exist an optimal weight matrix of Paley 17.\end{proof}

\section{Paley 17 delete a vertex}

After having shown that the inertia bound for Paley 17 is not tight, it is natural to consider its induced subgraphs.   In fact deleting a vertex from Paley 17 yields a graph which is also not inertia tight.

\begin{theorem}\label{thm:paley17minus0} The inertia bound for the graph obtained from Paley 17 by deleting a vertex is not tight.
\end{theorem} 

\begin{proof}  The argument is very similar to that of Paley 17.  For the sake of brevity, we include only an outline.  As Paley 17 is vertex transitive, the graph obtained by deleting the zero vertex is isomorphic to the graph obtained by deleting any other single vertex.  Let $G$ be Paley 17 less the zero vertex.  It can be shown that $\alpha(G)=3$ and that $G$ is $\alpha$-critical.   Thus by Observation \ref{obs:noedgedeletion} all entries of an optimal weight matrix of $G$ corresponding to edges are nonzero.
	 
	 There are two key properties of the graph $G_1$ in Section \ref{sec:indsub}.  First, if the entries of a weight matrix corresponding to edges are nonzero, then the determinant is nonzero.  Second, the sign of the determinant (and hence whether the inertia is $(4,3,0)$ or $(3,4,0)$) is determined by the sign of the triangle.  Of the 68 occurrences of the graph $G_1$ in Paley 17, only 40 are still present in $G$.  However, there are various other induced subgraphs of $G$ on 7 vertices which also possess the above two properties.   All but 8 of the 56 triangles of $G$ belong to such an induced subgraph.  (Every vertex of these 8 triangles is not adjacent to 0 in Paley 17.)   Setting aside these 8 triangles and using the same argument found in Lemma \ref{lem:triangleshavesamesign}, the other 48 triangles have the same sign in any optimal weight matrix for $G$.
	
	Suppose by way of contradiction that there exists an optimal weight matrix $M$ of $G$.  Using Lemma \ref{lem:path}, we may assume that all the entries of $M$ corresponding to the 1-edges of $G$ are positive.   The two 1-1-2 triangles which are not part of the 48 are $\Delta(5,6,7)$ and $\Delta(10,11,12)$.  Using the other 1-1-2 triangles, the signs of all the 2-edges of $G$, except $(5,7)$, $(10,12)$, $(1,16)$, are forced negative or positive depending on the desired sign of each triangle.  The two 2-2-4 triangles which are not part of the 48 are $\Delta(3,5,7)$ and $\Delta(10,12,14)$.  Using the other 2-2-4 triangles, the signs of all the 4-edges of $G$, except $(5,9)$, $(8,12)$, $(16,3)$, $(3,7)$, $(15,2)$, $(10,14)$, $(14,1)$, are determined.  The two 4-4-8 triangles which are not part of the 48 are $\Delta(3,7,11)$ and $\Delta(6,10,14)$.  Using the other 4-4-8 triangles we can determine the sign of two 8-edges, $(7,15)$ and $(2,10)$.  The two 8-8-1 triangles which are not part of the 48 are $\Delta(5,6,14)$ and $\Delta(3,11,12)$.  Since we know the sign of two of the 8-edges, we can use them and incident 1-edges, to determine the sign of more 8-edges using 8-8-1 triangles.  While at this stage we can't determine all of them, we can use these new 8-edges, to determine the sign of some 4-edges, using 4-4-8 triangles.	  In this manner it is possible to determine the sign of all edges of $G$ solely based on whether the 48 triangles have positive sign or negative sign.  
	
	In the case that the sign of the 48 triangles are positive, all the entries of $M$ corresponding to edges are determined to be positive.  Note that a copy of graph $G_2$ from Section \ref{sec:indsub} remains in $G$. In the case that the sign of the 48 triangles are negative all entries of $M$ corresponding to edges are negative except for all the 1-edges, and the edges $(16,1)$, $(16,3)$, $(15,2)$, $(14,1)$, $(16,7)$, $(15,6)$, $(14,5)$, $(13,4)$, $(12,3)$, $(11,2)$, $(10,1)$.  Using the copy of $G_2$ on vertices $1,2,3,4,7,13,14$, the corresponding principal submatrix will have positive determinant since the 5-cycle on vertices $1,2,4,13,14$ has negative sign.  Thus in either case $M$ has a principal submatrix with 4 positive eigenvalues and another with 4 negative eigenvalues.  Therefore by Lemma \ref{lem:4iscontradicting}, $M$ is not optimal, a contradiction.\end{proof}

\section{Conclusion}

The goal of this paper was to demonstrate that the answer to Question \ref{qu:the question} is no.  Having done that, we turn our attention to some related unanswered questions.

In \cite{rooney} a method for weighting Cayley graphs is introduced. This method weights the generators of the Cayley graph which corresponds to weighting the 1-regular and 2-regular subgraphs in its decomposition.  As Paley 17 has a 2-factorization, this method weights each of the 4 cycles of length 17.   Using this method it is possible to construct a weight matrix for Paley 17 with inertia $(13,4,0)$.   Assign the 1-edges a weight of 30, the 2-edges a weight of $-22$, the 4-edges a weight of $-12$, and the 8-edges a weight of 7.  Sage determined that this matrix has 13 positive eigenvalues and 4 negative eigenvalues.  Thus for Paley 17 the gap between $\alpha$ and $\min\{17-n_+(W),17-n_-(W)\}$ is one.

\begin{question} How large can the gap be in Equation \ref{eq:inertiabound}?
\end{question}  

It is somewhat unsatisfying that Paley 13 remains undetermined.  While considering this question we were able to show that any non-negative weight matrix for Paley 13 is not optimal.  

\begin{question} Does there exist an optimal weight matrix for Paley 13?
\end{question}

In \cite{elzinga} it is shown that all the induced subgraphs of Paley 13 are inertia tight.  
Theorem \ref{thm:paley17minus0} shows that at least one of the proper induced subgraphs of Paley is not inertia tight.  It seems likely that others are as well.

\begin{question} Which of the proper induced subgraphs of $P(17)$ are inertia tight?
\end{question}

We conclude with the following two questions.

\begin{question}  What other methods exist for determining whether a graph is inertia tight?
\end{question}

\begin{question} What other methods exist for constructing weight matrices with a relatively large number of positive or negative eigenvalues?
\end{question}


\bibliography{inertiabound}
\bibliographystyle{plain}
%
%
%
%
%

\end{document}